\documentclass[a4paper,12pt]{article}
\usepackage[T1]{fontenc}
\usepackage[utf8]{inputenc}
\usepackage{amsmath,amssymb}
\usepackage{amsthm}
\usepackage[nobysame]{amsrefs}
\usepackage{pifont}
\usepackage{xspace}
\usepackage{enumitem}
\usepackage{mathtools}
\usepackage[cal=boondoxo]{mathalfa}
\usepackage{xcolor}
\usepackage{stmaryrd}
\usepackage{tikz}
\usetikzlibrary{calc,positioning}

\usepackage{microtype}
\usepackage[unicode,hyperfootnotes=false,pdftex]{hyperref}
\usepackage{hyperxmp}
\usepackage{bookmark}
\usepackage{booktabs}

\hypersetup{%
        pdftitle={Sensitivity of rough differential equations: an approach through the Omega lemma},
        pdfauthor={Laure Coutin; Antoine Lejay},
        pdfdate={2017-12-12},
}

\SetEnumitemKey{thm}{noitemsep,topsep=0pt,leftmargin=2em,label={{\normalfont \Roman*)}},ref=\Roman*}
\SetEnumitemKey{hyp}{noitemsep,topsep=0pt,leftmargin=2em,label={{\normalfont \roman*)}},ref=\roman*}

\DeclarePairedDelimiterXPP{\normsup}[1]{}{\|}{\|}{_{\infty}}{#1}
\DeclarePairedDelimiterXPP{\normfp}[1]{}{\|}{\|}{_{\bullet p}}{#1}
\DeclarePairedDelimiterXPP{\normfq}[1]{}{\|}{\|}{_{\bullet q}}{#1}
\DeclarePairedDelimiterXPP{\normp}[1]{}{\|}{\|}{_{p}}{#1}
\DeclarePairedDelimiterXPP{\normq}[1]{}{\|}{\|}{_{q}}{#1}
\DeclarePairedDelimiterXPP{\normr}[1]{}{\|}{\|}{_{r}}{#1}
\DeclarePairedDelimiterXPP{\normfi}[2]{}{\|}{\|}{_{#1}}{#2}
\DeclarePairedDelimiterXPP{\normpot}[1]{}{\|}{\|}{_{p/2}}{#1}
\DeclarePairedDelimiterXPP{\normx}[1]{}{\|}{\|}{_{\bx}}{#1}
\DeclarePairedDelimiterXPP{\normfx}[1]{}{\|}{\|}{_{\bullet\bx}}{#1}
\DeclarePairedDelimiterXPP{\abs}[1]{}{|}{|}{}{#1}
\DeclarePairedDelimiterXPP{\normP}[1]{}{\|}{\|}{_{\uP}}{#1}
\DeclarePairedDelimiterXPP{\generalnorm}[2]{}{\|}{\|}{_{#1}}{#2}
\DeclarePairedDelimiter{\normop}{|}{|}
\DeclarePairedDelimiterXPP{\normhold}[2]{H_{#1}}{(}{)}{}{#2}

\newcommand{\okappa}{\overline{\kappa}}
\newcommand{\RR}{\mathbb{R}}
\newcommand{\vd}{\,\mathrm{d}}
\newcommand{\dD}{\mathrm{D}}
\newcommand{\fd}{\mathfrak{d}}
\newcommand{\bx}{\mathbf{x}}
\newcommand{\bb}{\mathbf{b}}
\newcommand{\Id}{\mathsf{Id}}
\newcommand{\fO}{\mathfrak{O}}

\newcommand{\fI}{\mathfrak{I}}
\newcommand{\fe}{\mathfrak{e}}
\newcommand{\fF}{\mathfrak{F}}
\newcommand{\ff}{\mathfrak{f}}
\newcommand{\fH}{\mathfrak{H}}
\newcommand{\fG}{\mathfrak{G}}
\newcommand{\fP}{\mathfrak{P}}
\newcommand{\cC}{\mathcal{C}}
\newcommand{\uU}{\mathrm{U}}
\newcommand{\uV}{\mathrm{V}}
\newcommand{\uB}{\mathrm{B}}
\newcommand{\uP}{\mathrm{P}}
\newcommand{\uPs}{\mathrm{P}^\star}
\newcommand{\uCs}{\mathrm{C}^\star}
\newcommand{\uPsim}{\mathrm{P}^\sim}
\newcommand{\uCsim}{\mathrm{C}^\sim}
\newcommand{\uQ}{\mathrm{Q}}
\newcommand{\uQs}{\mathrm{Q}^\star}
\newcommand{\uR}{\mathrm{R}}
\newcommand{\uT}{\mathrm{T}}
\newcommand{\uFi}{\mathrm{Fi}}
\newcommand{\uC}{\mathrm{C}}
\newcommand{\uG}{\mathrm{G}}
\newcommand{\uW}{\mathrm{W}}
\newcommand{\uX}{\mathrm{X}}
\newcommand{\zd}{z^\dag}
\newcommand{\yd}{y^\dag}
\newcommand{\ys}{y^\sharp}
\newcommand{\yf}{y^\flat}
\newcommand{\ystar}{y^*}

\newcommand\given{\nonscript\:\delimsize\vert\nonscript\:\mathopen{}} 
\newcommand\SetSymbol[1][]{\nonscript\:#1\vert\nonscript\:\mathopen{}\allowbreak}
\DeclarePairedDelimiterX\Set[1]\{\}{\renewcommand\given{\SetSymbol[\delimsize]}#1}
\DeclarePairedDelimiterX\Prb[1](){\renewcommand\given{\SetSymbol[\delimsize]}#1}

\newtheorem{lemma}{Lemma}
\newtheorem{theorem}{Theorem}
\newtheorem{corollary}{Corollary}
\newtheorem{proposition}{Proposition}
\theoremstyle{definition}

\newtheorem{notation}{Notation}

\theoremstyle{remark}
\newtheorem{remark}{Remark}

\begin{document}

\title{Sensitivity of rough differential equations:\\ an approach through the Omega lemma}

\author{Laure Coutin\footnote{Institut de Mathématiques de Toulouse; UMR5219\newline
	Université de Toulouse; UT3\newline
E-mail: \texttt{laure.coutin@math.univ-toulouse.fr}
} \and Antoine Lejay\footnote{
    Université de Lorraine, IECL, UMR 7502, Vand\oe uvre-lès-Nancy, F-54600, France\newline
    CNRS, IECL, UMR 7502, Vand\oe uvre-lès-Nancy, F-54600, France\newline
    Inria, Villers-lès-Nancy, F-54600, France\newline
E-mail: \texttt{Antoine.Lejay@univ-lorraine.fr}
}}
\date{December 12, 2017}

\maketitle

\begin{abstract}
The Itô map gives the solution of a Rough Differential Equation, a generalization
of an Ordinary Differential Equation driven by an irregular path, when existence
and uniqueness hold. By studying how a path is transformed through the vector
field which is integrated, we prove that the Itô map is Hölder or Lipschitz continuous 
with respect to all its parameters. This result unifies and weakens the hypotheses of the 
regularity results already established in the literature. 
\end{abstract}


%%%%%%%%%%%%%%%%%%%%%%%%%%%%%%%%%%%%%%%%%%%%%%%%%%%%%%%%%%%%%%%%%%%%%%
\section{Introduction}

The theory of rough paths is now a standard tool 
to deal with stochastic differential equations (SDE) driven 
by continuous processes other than the Brownian motion
such as the fractional Brownian motion. Even for standard
SDE, it has been proved to be a convenient tool
for dealing with large deviations or for numerical 
purposes. We refer the reader to 
\cites{lejay08b,lejay03a,friz-victoir,lyons06b,lyons02b,gubinelli,coutin12a,friz14a}
for a presentation of this theory with several points of view.
Here, we mainly rely on the notion of \emph{controlled rough path} of M. Gubinelli~\cites{gubinelli,friz14a}.

For a Banach space $\uU$, a time horizon $T>0$,
a regularity indice $p\in[2,3)$ and a control $\omega$ (See Section~\ref{sec:pvar} for a definition),
we denote by 
$\uC_p(\uU)$ the space of paths from $[0,T]$ to $\uU$ of finite $p$-variation with respect to the control $\omega$.
By this, we mean a path $x:[0,T]\to\uU$ 
such that with $\abs{x_t-x_s}\leq C\omega(s,t)^{1/p}$ for some constant $C$.

A \emph{rough path} $\bx$ is an extension in the non-commutative
tensor space $\uT_2(\uU):=1\oplus\uU\oplus(\uU\otimes\uU)$ of a path $x$ in $\uC_p(\uU)$. 
This extension $\bx$ is defined through algebraic and analytic 
properties. It is decomposed as $\bx:=1+\bx^1+\bx^2$ with $\bx^1:=x$ in $\uU$
and $\bx^2\in\uU\otimes\uU$. The increments of $\bx$ are defined by $\bx_{s,t}:=\bx_s^{-1}\otimes \bx_t$. 
It satisfies the multiplicative property $\bx_{r,t}=\bx_{r,s}\otimes\bx_{s,t}$ for any $0\leq r\leq s\leq t\leq T$.
The space  is equipped with the topology induced
by the $p$-variation distance with respect to $\omega$. 
The space of rough paths of finite $p$-variations 
with respect to the control~$\omega$ is denoted by  $\uR_p(\uU)$.
There exists a natural projection from $\uR_p(\uU)$ onto $\uC_p(\uU)$. 
Conversely, a path may be lifted from $\uC_p(\uU)$ to $\uR_p(\uU)$ \cite{lyons-victoir}, 
yet this cannot be done canonically. 

Given a rough path $\bx\in\uR_p(\uU)$
and a vector field $f:\uV\to L(\uU,\uV)$ for a Banach space $\uV$, 
a controlled differential equation 
\begin{equation}
\label{eq-rde1}
y_t=a+\int_0^t f(y_s)\vd\bx_s
\end{equation}
is well defined provided that $f$ is regular enough.
This equation is called a \emph{rough differential equation} (RDE).
For a smooth path $x$, a rough path $\bx$ could be naturally 
constructed using the iterated integrals of $x$. In this case, the solution 
to \eqref{eq-rde1} corresponds to the solution to 
the ordinary differential equation $y_t=a+\int_0^t f(y_s)\vd x_s$.
The theory of rough paths provides us with natural extension of 
controlled differential equations.

When \eqref{eq-rde1} has a unique solution $y\in\uC_p(\uV)$ 
for any $\bx\in\uR_p(\uU)$ given that the vector field $f$ belongs to a proper
subspace $\uFi$, the map $\fI:(a,\bx,f)\mapsto y$ from $\uU\times\uR_p(\uU)\times \uFi$
to $\uC_p(\uV)$ is called the \emph{Itô map}. 
The Itô map is actually locally Lipschitz continuous on $\uU\times\uR_p(\uU)\times\uFi$ 
when $\uFi$ is equipped with the proper topology \cites{lejay1,lejay2,friz-victoir}. 

Together with the Itô map $\fI$, we could consider for each $t\in[0,T]$ 
$\ff_t(a,\bx,f)=\fe_t\circ \fI(a,\bx,f)$ from $\uU\times \uR_p(\uU)\times \uFi$ to $\uV$,
where $\fe_t$ is the \emph{evaluation map} $\fe_t(x)=x(t)$.
The family $\{\ff_t(a,\bx,f)\}_{t\in[0,T]}$ is the \emph{flow} associated to the RDE~\eqref{eq-rde1}.
For ordinary differential equations, the flow defines a family of homeomorphisms or 
diffeomorphisms.

The differentiability properties of the Itô map or the flow are
very important in view of applications.  For SDE, Malliavin calculus opens the door 
to existence of a density and its regularity \cites{nualart-book,malliavin}, 
large deviation results
\cite{inahama1}, Monte Carlo methods \cites{malliavin,gobet},~...

In this article, we then consider the differentiability properties,
understood as Fréchet differentiability, 
of the Itô map under minimal regularity conditions on the vector field. 
Therefore, we extend the current results by proving Hölder continuity
of the Itô map. This generalizes to $2\leq p<3$ the results
of T.~Lyons and X.~Li \cite{li05a}.

Several strategies have been developed to consider the regularity of the Itô map and flows, which 
we review now. For this, we need to introduce some notations. 
\begin{itemize}[noitemsep,leftmargin=0pt,itemindent=1em,topsep=0pt,partopsep=0pt,after={\vspace{\medskipamount}},label={$\star$}]
    \item The space of \emph{geometric rough paths} $\uG_p(\uU)$ is roughly
	described as the limit of the natural lift of smooth rough paths
	through their iterated integrals. 

    \item For a path $\bx\in\uR_p(\uU)$ and a Banach space $\uV$, $\uP_p(\bx,\uV)$ is
	the space of \emph{controlled rough path} (CRP). It contains paths
	$z:[0,T]\to\uV$ whose increments are like those of $\bx$, that is
	$z_{s,t}=\zd_s x_{s,t}+z^\sharp_{s,t}$ for proper
	$(\zd,z^\sharp)$ (for a formal definition, see
	Section~\ref{sec:crp}).  A CRP $z$ is best identified with the pair
	$(z,\zd)$ (See Remark~\ref{rem:crp} in Section~\ref{sec:crp}).

    \item For a path $\bx\in\uR_p(\uU)$ and $h\in\uC_q(\uU)$ with $1/p+1/q>1$
	(which means that $1\leq q<2)$, there exists a natural way to construct
	a rough path $\bx(h)\in\uR_p(\Omega)$ such that
	$\pi(\bx(h))=\pi(\bx)+h$, where $\pi$ is the natural projection from
	$\uT_2(\uU)$ onto $\uU$, and $\bx(0)=\bx$. 
	The map $h\mapsto \bx(h)$ is $\cC^\infty$-Fréchet from $\uC_q(\uU)$
	to $\uR_p(\uU)$.

    \item A vector field $f:\uV\to L(\uU,\uV)$ is said to be
	\emph{$\gamma$-Lipschitz} ($\gamma>0$) if it is differentiable up to
	order $\lfloor\gamma\rfloor$ with a derivative of order $\lfloor
	\gamma\rfloor$ which is $(\gamma-\lfloor \gamma\rfloor)$-Hölder
	continuous (See \cite{friz-victoir}, p.~213). 
\end{itemize}

The flow and differentiability properties have already been dealt with in the following articles:
\begin{itemize}[noitemsep,leftmargin=0pt,itemindent=1em,topsep=0pt,partopsep=0pt,after={\vspace{\medskipamount}},label={$\star$}]
    \item In \cite{lyons97b}, T. Lyons and Z. Qian have 
	studied the flow property for solutions
	to $y_t=a+\int_0^t f(y_s)\vd\bx_s+\int_0^t g(y_s)\vd h_s$
	for a ``regular'' path $h$ subject to a perturbation for $\uV=\RR^d$.

    \item In \cite{lyons-qian}, T. Lyons and Z. Qian 
	showed that the Itô map provides a flow of 
	diffeomorphisms when the driving rough path is geometric for $\uV=\RR^d$.

    \item T. Lyons and Z. Qian \cite{lyons97c} and
	more recently Z. Qian and J.~Tudor \cite{qian}
	have studied the perturbation of the Itô map 
	when the rough path is perturbed by a regular
	path $h$ and the structure of the tangent spaces for finite dimensional Banach space.
       	As these constructions hold in tensor spaces, quadratic terms are involved. They lead
	to rather intricate expressions.

   \item The properties of flow also arise directly from the 
	constructions from almost flows, in the approach 
	from I.~Bailleul \cite{bailleul} in possibly infinite dimensional Banach spaces. 
	Firstly, only Lipschitz flow 
	were considered, but recently in~\cite{bailleul15b}, it was extended to Hölder
	continuous flows. 

    \item In the case $1\leq p<2$, T.~Lyons and Z.~Li proved in \cite{li05a} (see also \cite{lejay3}) that 
	\begin{equation*}
	x\in\uC_p(\uU)\mapsto \fI(a,x,f)\in\uC_p(\uV)
	\text{ is locally $\cC^{k}$-Fréchet differentiable}
	\end{equation*}
	provided that $\uU$ and $\uV$ are finite dimensional Banach spaces
       and $f$ is $\cC^{k+\alpha+\epsilon}$, $k\geq 1$, $\alpha\in(p-1,1-\epsilon)$, $\epsilon\in(0,1)$.

    \item In the book of P.~Friz and N.~Victoir \cite{friz-victoir}, 
    it is proved that 
    \begin{gather*}
    (b,h)\in\uU\times\uC_q(\uU)\mapsto \fI(a+b,\bx(h),f)\in\uC_p(\uU)\\
    \text{ is locally $\cC^k$-Fréchet at }(a,\bx)\in\uU\times\uG_p(\uU)
    \end{gather*}
    provided that $\bx\in\uG_p(\uU)$, $\uU$ is finite-dimensional and $f$
    is of class $\cC^{k-1+\gamma}(\uU,\uV)$
    with $\gamma>p$ and $k\geq 1$
    (See \cite{friz-victoir}*{Theorem~11.6, p.~287}).
    It is also proved that $\ff:[0,T]\times\uU\to \uV$ is a flow 
    of $\cC^k$-diffeomorphisms, and that $\ff$ and its derivatives are
    uniformly continuous with respect to $\bx\in\uG_p(\uU)$
    (See \cite{friz-victoir}*{Section~11.2, p.~289}). 
    Transposed in the context of CRP in \cite{friz14a}, 
    the flow $a\mapsto \ff_t(a,\bx,f)$ is locally a diffeomorphism of class  
    $\cC^{k+1}$ for a vector field $f$ is $\cC^{k+3}$-Fréchet.

% Inahama & Kawabi
\item In a series of articles \cites{inahama1,inahama07b,inahama4,inahama5} (See also \cite{zhang12a}), 
Y. Inahama and H. Kawabi have studied various aspects of stochastic Taylor developments in $\epsilon$ of solutions
to 
\begin{equation}
    \label{eq:ik:1}
y^\epsilon_t=a+\int_0^t f(y^\epsilon_s)\vd \fd_\epsilon\bx_s +\int_0^t g(y_s^\epsilon)\vd h_s
\end{equation}
around the solutions to $z_t=a+\int_0^t g(z_s)\vd h_s$, when 
$\bx\in\uR_p(\uW)$, $h\in\uC_q(\uU)$ with $1/p+1/q>1$, $f:\uV\to L(\uW,\uV)$, $g:\uV\to L(\uU,\uV)$
and $\fd_\epsilon:\uR_p(\uW)\to\uR_p(\uW)$ is the dilatation operator defined by $\fd_\epsilon\bx:=1+\epsilon\bx^1+\epsilon^2\bx^2$.
Up to the natural injection of $\bx$ to $\uR_p(\uW\oplus\uU)$, $h$ to $\uC_p(\uW\oplus\uU)$, 
$f:\uV\to L(\uW\oplus\uU,\uV)$ and $g:\uV\to L(\uW\oplus\uU,\uV)$, \eqref{eq:ik:1} is recast as 
\begin{equation*}
y^\epsilon=\fI(a,\fd_\epsilon\bx(h),f+g).
\end{equation*}

% Bailleul
\item Using a Banach space version of the Implicit Functions Theorem, 
    I. Bailleul recently proved in \cite{bailleul15} that 
    \begin{gather*}
    (z,f)\in\uP_p(\bb,\uU)\times \cC^k(\uU,\uV)\mapsto \fI(a,\fP_\bb(z),f)\in\uP_p(\bb,\uV)\\
    \text{is $\cC^{\lfloor k\rfloor-2}$-Fréchet differentiable}
    \end{gather*}
    provided that $\bb\in\uG_p(\uB)$ and $k\geq 3$, where $\fP_\bb(z)$ 
    lifts $z\in\uP_p(\bb,\uU)$ to a geometric rough path. 

\item SDE driven by fractional Brownian motion and Gaussian processes attracted
    a lot of attention
    \cites{cass3,inahama2,marie,cass,besalu,decreusefond,baudoin2,baudoin3,baudoin4,hu,leon,deya,chronopoulou,nualart,inhama8,baudoin5}.
    Since integrability is the key to derive some integration by parts formula
    in Malliavin type calculus, several articles deal with moments estimates
    for solutions to linear equations driven by Gaussian rough paths
    \cites{baudoin,inahama3,friz1,unterberger,cass2,hairer}. 
\end{itemize} 

In this article, 
\begin{itemize}[noitemsep,leftmargin=0pt,itemindent=1em,topsep=0pt,partopsep=0pt,after={\vspace{\medskipamount}},label={$\star$},itemsep=\smallskipamount]

    \item We establish that the Itô map $\fI$ is locally of class $\cC^{\gamma-\epsilon}$ 
	for all the types of perturbations of the driving rough path $\bx$ seen above. 
	We then generalized the results in \cite{bailleul15}, 
	\cite{friz14a}*{Sect.~8.4} or \cite{friz-victoir}
	by providing the Hölder regularity, and not only differentiability.
	At the exception of  the work of T.~Lyons and X.D.~Li~\cite{li05a} 
	for the Young case ($1\leq p<2$), at the best of our knowledge, none of the works cited
	above deal with the Hölder regularity of the derivatives of the Itô map.
       The cited results are proved under stronger regularity conditions than ours on the vector field. 
       In~\cite{bailleul15}, the starting point is kept fixed. 
	While in~\cites{friz14a,friz-victoir}, only the regularity with respect
	to the starting point and perturbation of the form $\bx(h)$ of the driver are considered, 
	as in~\cite{li05a}.
	Unlike~\cite{friz-victoir}, the chain rule may be applied as 
	our solutions are constructed as CRP.

    \item By adding more flexibility in the notion of CRP,
	we define a bilinear continuous integral with CRP 
	both as integrand and integrators.
	This simple trick allows one to focus on the effect of a non-linear
	function applied to a CRP and weaken the regularity assumptions
	imposed on the vector fields in~\cites{friz14a,bailleul15}.

    \item Following the approach of \cite{abraham}, we provide
	a version of the so-called \emph{Omega lemma} for paths
	of finite $p$-variations with $1\leq p<2$ and for CRP (for $2\leq p<3$). 
	For a function~$f$ of given regularity, this lemma states
	the regularity of the map $\fO f:y\mapsto \Set{f(y_t)}_{t\in[0,T]}$ when $y$ is a path 
	of finite $p$-variation or a CRP on $[0,T]$. 
	When dealing with continuous paths with the sup-norms, where $\fO f$ has
	the regularity of $f$ (for functions of class~$\cC^k$, but also Hölder or Sobolev).
	When dealing with paths of finite $p$-variation, the situation is more cumbersome. 
        This explains the losses in the regularity observed in \cite{li05a}.

    \item We provide a ``genuine rough path'' approach which removes the restriction 
	implied by smooth rough paths, the restriction to geometric rough paths (which could 
	be dealt otherwise with $(p,p/2)$-rough paths~\cite{lejay-victoir}) as well as any 
	restriction on the dimensions of the Banach spaces~$\uU$ and~$\uV$.
	Although we use a CRP, the Duhamel formula shown in \cite{coutin} could
	    serve to prove a similar result for (partial) rough paths and not
	only CRP. 

    \item We exemplify the difference between the ``rough situation''
	and the smooth one. For ordinary differential equation, the spirit
	of the Omega lemma together with the Implicit Functions Theorem 
	is that the regularity of the vector field is transported 
	into the regularity of the Itô map. Once the Omega lemma
	is stated for our spaces of paths, its implication on the regularity
       of the Itô map is immediate.
\end{itemize}

%%%%%%%%%%%%%%%%%%%%%%%%%%%%%%%%%%%%%%%%%%%%%%%%%%%%%%%%%%%%%%%%%%%%%%
\section{Notations}

Througout all the article, we denote by $\uU$, $\uV$ and $\uW$ Banach spaces. 

For two such Banach spaces $\uU$ and $\uV$, we denote by $L(\uU,\uV)$ the set of linear \emph{continuous}
maps from $\uU$ to $\uV$. If $A\in L(\uU,\uV)$ is invertible, then its inverse $A^{-1}$ is itself continuous
and thus belongs to $L(\uV,\uU)$.

For a functional $\fF$ from $\uU$ to $\uV$, we denote by $\dD\fF$
its Fréchet derivative, which is a map from $\uV$ to $L(\uV,\uU)$, and by $\dD_y\fF$
its Fréchet derivative in the direction of a variable $y$
in $\uU$.

For a function $f$ and some $\alpha\in(0,1]$, we denote by $H_\alpha(f)$
its \emph{$\alpha$-Hölder semi-norm} $H_\alpha(f):=\sup_{x\not=y} \abs{f(x)-f(y)}/\abs{x-y}^\alpha$.

For any $\alpha>0$, we denote by $\cC^{\alpha}(U,\uV)$  the set of continuous, bounded functions
from~$U$ to~$\uV$ with, bounded continuous (Fréchet) derivatives up to order 
$k:=\lfloor \alpha\rfloor$, 
and its derivative of order $k$ is $(\alpha-k)$-Hölder continuous.
A function in $\cC^{\alpha}(U,\uV)$ is simply said to be of class $\cC^\alpha$ (in this article, we restrict
ourselves to non-integer values of $\alpha$).

With the convention that $\dD^0f:=f$, we write for $f$ in $\cC^{\alpha}(U,\uV)$, 
\begin{equation*}
    \normfi{\alpha}{f}:=\max_{j=0,\dotsc,k}\Set*{\normsup{\dD^j f},\normhold{\alpha-k}{\dD^k f}}.
\end{equation*}

Our result are stated for bounded functions $f$. For proving existence and
uniqueness, this boundedness condition may be relaxed by keeping only the
boundedness of the derivatives of $f$ (See \textit{e.g.},
\cites{lejay1,friz-victoir,lejay2},~...).  Once existence is proved under this
linear growth condition, there is no problem in assuming that $f$ itself is
bounded since we only use estimates locally.  This justifies our choice for the
sake of simplicity.

%%%%%%%%%%%%%%%%%%%%%%%%%%%%%%%%%%%%%%%%%%%%%%%%%%%%%%%%%%%%%%%%%%%%%%
\section{The Implicit Functions Theorem}

Let us consider two Banach spaces $\uP$ and $\Lambda$  as
well as a functional $\fF$ from $\uP\times\Lambda$ to $\uP$.
Here, $\uP$ plays the role of the spaces of paths, while
$\Lambda$ is the space of parameters of the equation. 

We consider first solutions $y$ to the fixed point problem 
\begin{equation}
    \label{eq:121}
y=\fF(y,\lambda)+b,\ (\lambda,b)\in\Lambda\times\uP.
\end{equation}
This is an abstract way to consider equations of type
$y_t=a+\int_0^t f(y_s)\vd\bx_s+b_t$, whose parameters are $\lambda=(a,f,\bx)\in\Lambda$
and $b\in\uP$.

To ensure uniqueness of the solutions to \eqref{eq:121}, we slightly change the problem.
We assume that $\uP$ contains a Banach sub-space $\uPs$,
typically, the paths that start from $0$.
As $\uPs$ is stable under addition, we consider the 
quotient space~$\uPsim:=\uP/\uPs$ defined for the equivalence relation $x\sim y$
when~$x-y\in\uPs$. This quotient is nothing more than a way to encode the starting point. 

We now consider instead of \eqref{eq:121} the problem 
\begin{equation}
    \label{eq:122}
y^*=\fF(y^*+z,\lambda)-z+b,\ b\in\uPs,\ z\in\uPsim,\ \lambda\in\Lambda.
\end{equation}
There is clearly no problem in restricting $b$ to $\uPs$, 
since otherwise one has to change $\lambda$ and $z$ accordingly. Solving \eqref{eq:122}
implies that \eqref{eq:121} is solved for~$y=y^\star+z$.

For an integer $k\geq 0$ and $0<\alpha\leq 1$,
if $\fG(y^*,z,\lambda):=\fF(y^*+z,\lambda)$ and $\fF$ is of class $\cC^{k+\alpha}$ with respect to $(y,\lambda)$, 
then $\fG$ is of class $\cC^{k+\alpha}$ with respect to~$(y^*,z,\lambda)$

The reason for considering \eqref{eq:122} instead of~\eqref{eq:121}
is that for the cases we consider, $\fG$ will be strictly contractive
in $y^*$, ensuring the existence of the unique solution to~\eqref{eq:122}.

We use the following version of the Implicit Functions Theorem
(See \textit{e.g.}~\cite{abraham}*{\S~2.5.7, p.~121} for a $\cC^k$-version
of the Implicit Functions Theorem\footnote{To extend it to $\cC^{k+\alpha}$-Hölder
continuous functions, we have just to note that the derivative 
of $(b,z,\lambda)\mapsto (\fH(b,z,\lambda),z,\lambda)$ below, which is given by the inverse function theorem, 
is the composition of $z\mapsto z^{-1}$, which 
is $\cC^{\infty}$, $\fG$, which is~$\cC^{k-1+\alpha}$ and $\fH$ which
is~$\cC^{k}$, see~\cite{norton}).}.

\begin{theorem}[Implicit Functions Theorem]
    \label{thm:implicit}
    Let us assume that 
    \begin{enumerate}[leftmargin=2em,label={\roman*)}]
	\item The map 
    $\fG(y^*,z,\lambda):=\fF(y^*+z,\lambda)$ is of class $\cC^{k+\alpha}$
    from $\uX:=\uPs\times \uPsim\times\Lambda$ to $\uPs$
   for $k\geq 1$, $0<\alpha\leq 1$ with respect to $(y^*,z,\lambda)\in\uX$. 

   \item For some $(\widehat{y}^*,\widehat{z},\widehat{\lambda})\in\uX$ and any $b^*\in\uPs$, 
       the exists a unique solution $h^*$ in $\uPs$ to 
       \begin{equation*}
       h^*=\dD_{y^*}\fG(\widehat{y}^*,\widehat{z},\widehat{\lambda})(h^*)+b^*
       \end{equation*}
       with $\normP{h^*}\leq C\normP{b^*}$ for some constant $C\geq 0$. 
       This means that $\Id-\dD_{\widehat{y}}\fG(\cdot,\widehat{z},\widehat{\lambda})$
       is invertible from $\uPs$ to $\uPs$ with a bounded inverse.
   \end{enumerate}
   Then there exists a neighborhood $U$ of $(\widehat{z},\widehat{\lambda})\in\uPsim\times\Lambda$,
   a neighborhood $V$ of $\fG(\widehat{y}^*,\widehat{z},\widehat{\lambda})$, as well as a unique
   map $\fH$ from $V\times U$ to $\uPs$ which solves 
   \begin{equation*}
       \fH(b,z,\lambda)=\fG(\fH(b,z,\lambda),z,\lambda)+b,\ \forall (b,z,\lambda)\in V\times U.
   \end{equation*}
   In other words, $\fI(b,z,\lambda):=z+\fH(b,z,\lambda)$ is locally the solution 
   in $z+\uPs$ to $\fI(b,z,\lambda)=-z+\fF(\fI(b,z,\lambda),\lambda)+b$.
\end{theorem}

\begin{remark}
    \label{rem:implicit}
Actually, we do not use this theorem in this form. We show 
that~$\dD_y \fG(\cdot,z,\lambda)$ is contractive when restricted
to a bounded, closed, convex set~$C$ of $\uPs\times\uPsim\times\Lambda$, and only 
on a time interval $\tau$ which is small enough, in function 
of the radius of $C$. The controls we get allow us to solve iteratively the equations on abutting time intervals
$\tau_i$ and to ``stack them up'' to get the result on
any finite time interval (and even globally for suitable vector fields).
As for this, we have only to re-use with slight adaptations 
what is already largely been done, we do not treat these issues. 
\end{remark}

%%%%%%%%%%%%%%%%%%%%%%%%%%%%%%%%%%%%%%%%%%%%%%%%%%%%%%%%%%%%%%%%%%%%%%
%%%%%%%%%%%%%%%%%%%%%%%%%%%%%%%%%%%%%%%%%%%%%%%%%%%%%%%%%%%%%%%%%%%%%%
%%%%%%%%%%%%%%%%%%%%%%%%%%%%%%%%%%%%%%%%%%%%%%%%%%%%%%%%%%%%%%%%%%%%%%
%%%%%%%%%%%%%%%%%%%%%%%%%%%%%%%%%%%%%%%%%%%%%%%%%%%%%%%%%%%%%%%%%%%%%%

%%%%%%%%%%%%%%%%%%%%%%%%%%%%%%%%%%%%%%%%%%%%%%%%%%%%%%%%%%%%%%%%%%%%%%
\section{The Omega lemma for paths of finite $p$-variation}
%%%%%%%%%%%%%%%%%%%%%%%%%%%%%%%%%%%%%%%%%%%%%%%%%%%%%%%%%%%%%%%%%%%%%%

\label{sec:pvar}

We consider a time horizon $T>0$. A \emph{control} $\omega$
is a non-negative function defined on sub-intervals $[s,t]\subset[0,T]$ 
which is super-additive and continuous close to the diagonal $\Set{(t,t)\given t\in[0,T]}$.
This means that is $\omega_{r,s}+\omega_{s,t}\leq \omega_{r,t}$
for $0\leq r\leq s\leq t\leq T$.

For a path $x$ from $[0,T]$ to $\uV$, we set $x_{s,t}:=x_{[s,t]}:=x_t-x_s$.
For some~$p\geq 1$, a path $x$ of finite $p$-variation controlled by $\omega$
satisfies
\begin{equation*}
    \normp{x}:=\sup_{\substack{[s,t]\subset[0,T]\\s\not= t}}\frac{\abs{x_t-x_s}}{\omega_{s,t}^{1/p}}.
\end{equation*}
We denote by $\uC_p(\uV)$ the space of such paths, which is a Banach 
space with the norm 
\begin{equation*}
    \normfp{x}:=\abs{x_{0}}+\normp{x}.
\end{equation*}
The space $\uC_p(\uV)$ is continuously embedded in the space
of continuous functions~$\uC(\uV)$ with the sup-norm $\normsup{\cdot}$, with 
\begin{equation}
    \label{eq:70}
    \normsup{x}\leq \abs{x_{0}}+\normp{x}\omega_{0,T}^{1/p}.
\end{equation}
For any $q\geq p$, $\uC_p(\uV)$ is also continuously embedded
in~$\uC_q(\uV)$.

We call a \emph{universal constant} a constant that depends
only on $\omega_{0,T}$ and the parameters $p$, $q$, $\kappa$, $\gamma$,  ...
that will appear later.

\begin{proposition}[{L.C. Young \cite{young36a}}] 
    \label{prop:young}
    Let $p,q\geq 1$ such that $1/p+1/q>1$.
	There exists a unique continuous, bilinear map 
	\begin{equation*}
	    \begin{aligned}
	    \uC_p(\uU)\times\uC_q(L(\uV,\uU))&\to \uC_p(\uV)\\
		(x,y)&\mapsto \int_0^\cdot y\vd x
	    \end{aligned}
	\end{equation*}
	which satisfies $\int_0^0 y_r\vd x_r=0$  for any $(x,y)$ and any $[s,t]\subset[0,T]$,
	\begin{gather}
	    \label{eq:115}
	    \abs*{\int_s^t y_r\vd x_r-y_sx_{s,t}}
	    \leq K\normq{y}\normp{x}\omega_{s,t}^{\frac{1}{p} + \frac{1}{q}}
	\end{gather}
	for some universal constant $K$. 
\end{proposition}

For some $\kappa\in[0,1]$, we set $\overline{\kappa}:=1-\kappa$.

\begin{lemma}
\label{lem-1}
Let $g\in\cC^\gamma(\uV,\uW)$. Then for any $\kappa\in[0,1]$ and $\gamma\in[0,1]$, 
\begin{multline}
\label{eq-20}
\abs{g(z)-g(y)-g(z')+g(y')}\\
\leq \normhold{\gamma}{g}(\abs{y'-y}^{\kappa\gamma}+\abs{z'-z}^{\kappa\gamma})
(\abs{z'-y'}^{\gamma\overline{\kappa}}+\abs{z-y}^{\gamma\overline{\kappa}})
\end{multline}
for all $y,z,y',z'\in\uV$.
\end{lemma}

\begin{proof} First, 
    \begin{equation*}
\abs{g(z)-g(y)-g(z')+g(y')}
	\leq \normhold{\gamma}{g}(\abs{y'-y}^\gamma+\abs{z'-z}^\gamma).
    \end{equation*}
    By inverting the roles of $z'$ and $y$ in the above equation, we get a similar inequality. 
    Choosing $\kappa\in[0,1]$ and raising the first inequality to power $\kappa$
    and the second one to power $\overline{\kappa}$ leads to the result.
\end{proof}

We now fix $p\geq 1$, $\kappa\in(0,1)$ and $\gamma\in(0,1]$. 
We set $q:=p/\kappa\gamma$. We define
\begin{equation*}
    \uCs_p(\uV):=\Set{y\in\uC_p(\uV)\given y_{0}=0}
    \text{ and }\uCs_q(\uV):=\Set{y\in\uC_q(\uV)\given y_{0}=0}.
\end{equation*}

An immediate consequence of this lemma is that for $\kappa\in(0,1)$, $\gamma\in(0,1]$,
the map $\fO g:y\mapsto \Set{g(y_t)}_{t\in[0,T]}$ is $\gamma\overline{\kappa}$-Hölder continuous
from $\uCs_p(\uV)$ to $\uCs_{p/\kappa\gamma}$ with Hölder constant 
$2^\gamma\normhold{\gamma}{g}\omega_{0,T}^{\gamma\overline{\kappa}/p}$
when $g$ is $\gamma$-Hölder continuous.

We now consider the case of higher differentiability of $g$.

We now give an alternative proof of the one 
in \cite{li05a}*{Theorems~2.13 and~2.15}, which  mostly differs in the use
of the converse of Taylor's theorem.
For CRP in Section~\ref{sec:crp}, the proof of Proposition~\ref{prop:om:crp}
will be modelled on this one.

We still use the terminology of~\cite{abraham} regarding the
Omega lemma. The Omega operator~$\fO$ transforms a function 
$f$ between two Banach spaces $\uU$ and~$\uV$ to a function
mapping continuous paths from $[0,T]$ to $\uU$ 
to continuous paths from $[0,T]$ to $\uV$. The idea is then 
to study the regularity of $\fO f$  in function of 
the regularity of $f$ and the one of the paths that 
are carried by $\fO f$. The difference with the results 
in \cite{abraham} is that we use the $p$-variation norm 
instead of the sup-norm. This leads to a slight loss of regularity
when transforming $f$ to $\fO f$.

\begin{proposition}[The Omega lemma for paths of finite $p$-variation]
    \label{prop:om:fpv}
    For $p\geq 1$, $k\geq 1$, $\gamma\in(0,1]$, $\kappa\in(0,1)$ 
and $f$ of class $\cC^{k+\gamma}$ from $\uV$ to $\uW:=L(\uU,\uV)$,
$\fO f(y):=(f(y_t))_{t\in[0,T]}$ 
   is of class $\cC^{k+\overline{\kappa}\gamma}$ from any ball of radius $\rho>0$ of $\uC_p(\uV)$
   to $\uC_q(\uW)$ with $q:=p/\kappa\gamma$.
   Besides, $\dD_y\fO f\cdot h=(\dD f(y_t)\cdot h_t)_{t\in[0,T]}\in\uC_q(\uW)$
   for any $y,h\in\uC_p(\uV)$. 
   Finally, $\dD\fO(y)\cdot h\in\uCs_q(\uW)$ when $h\in\uCs_p(\uV)$.
\end{proposition}
\begin{proof}
    \noindent I) Assume that $f\in\cC^\gamma$.
       Thanks to the embedding from $\uC_{p/\gamma}(L(\uU,\uV))$ to $\uC_q(L(\uU,\uV))$,
	$\fO f$ maps $\uCs_p(\uV)$ to $\uC_{q}(L(\uU,\uV))$ 
    with $\normq{\fO f(y)}\leq C\normhold{\gamma}{f}\normp{y}$
   for all $y\in\uCs_p(\uV)$.  

    With Lemma~\ref{lem-1}, we easily obtain that
    \begin{equation*}
	\normfq{\fO f(y)-\fO f(z)}\leq \normhold{\gamma}{f}(1+\omega_{0,T}^{\okappa\gamma/p})\normfp{y-z}^{\okappa\gamma}
	(\normp{y}^{\kappa\gamma}+\normp{z}^{\kappa\gamma}).
    \end{equation*}
    Then $\fO f$ is locally $\overline{\kappa}\gamma$-Hölder continuous
    from $\uC_p(\uV)$ to $\uC_q(\uW)$.

    \medskip\noindent II)  For some Banach space $\uW'$, 
    if $y\in\uC_q(L(\uV\otimes\uW',\uV))$ 
    and $z\in\uC_p(\uV)$ (resp. $y\in\uC_q(\uV)$, $z\in\uC_p(\uU)$), 
    it is straightforward to 
    show with \eqref{eq:70} that $yz\in\uC_q(L(\uW',\uV))$ (resp. $y\otimes z\in\uC_q(\uV\otimes\uU)$)
    with 
    \begin{equation*}
	\normq{y\cdot z}\leq (1+2\omega_{0,T}^{1/q})\normfq{y}\normfq{z}
	\text{ and }\normfq{y\cdot z}\leq \abs{y_{0}}\cdot\abs{z_{0}}+\normq{y\cdot z},
    \end{equation*}
    where $y\cdot z=(y_tz_t)_{t\in[0,t]}$ (resp. $y\cdot z=(y_t\otimes z_t)_{t\in[0,T]}$)
    Besides, if $z_0=0$, then $(yz)_0=0$ (resp. $(y\otimes z)_0=0$).

    \medskip\noindent III) Let us assume now that $f\in\cC^{k+\gamma}$
    for some $k\geq 1$. With the Taylor development of $f$ up
    to order $k$, 
    \begin{gather*}
    f(y+z)=f(y)+\sum_{i=1}^k \frac{1}{i!}\dD^i f(y)z^{\otimes k}
    +R(y,z)
    \shortintertext{with}
    R(y,z)=
    \int_0^1 \frac{(1-s)^{k-1}}{(k-1)!}(\dD^k f(y+s z)-\dD^k f(y))z^{\otimes k}\vd s.
    \end{gather*}
    Since $\dD^i f$ is of class $\cC^{k-i+\gamma}$
    from $\uV$ to $L(\uV^{\otimes i},L(\uU,\uV))$ identified with $L(\uV^{\otimes i}\otimes\uU,\uV)$,
    then $y\in\uC_p(\uV)\mapsto (\dD^i f(y_t))_{t\in[0,T]}$ 
    takes its values in $\uC_q(L(\uV^{\otimes i}\otimes \uU,\uV))$.

    For $y,z^{(1)},\dotsc,z^{(i)}\in\uC_p(\uV)$, write 
    \begin{equation*}
	\phi_i(y)\cdot z^{(1)}\otimes\dotsc\otimes z^{(i)}
	:=(\dD^i f(y_t)\cdot z^{(1)}\otimes\dotsc\otimes z^{(i)})_{t\in[0,T]}.
    \end{equation*}
    Since $\uC_p(\uV)$ is continuously embedded in $\uC_q(\uV)$, 
    II) implies that $\phi_i(y)$ is multi-linear and continuous from $\uC_p(\uV)^{\otimes i}$
    to $\uC_q(\uW)$.

    Similarly, for some constant $C$ that depends only on $\rho$ (the radius of the ball such that 
    $\normfp{y}\leq \rho$), $p$, $q$ and $\omega_{0,T}$,  
    \begin{equation*}
	\frac{\normfq{R(y,z)}}{\normfp{z}^k}
	\leq 
	C\normhold{\gamma}{\dD^k f} \normfp{z}^{\overline{\kappa}\gamma}.
    \end{equation*}
    The converse of the Taylor's theorem \cite{albrecht} 
    implies that $\fO f$ is locally of class~$\cC^{k}$ 
    from $\uC_p(\uV)$ to $\uC_q(\uW)$. In addition, it is easily shown
    from I) that since $\dD^k f$ is locally of class $\cC^\gamma$, 
    $\fO f$ is locally of class $\cC^{k+\overline{\kappa}\gamma}$
    from $\uC_p(\uV)$ to $\uC_q(\uW)$.
\end{proof}
%%%%

%%
Given $0<\kappa<1$ with $1+\kappa\gamma>p$, we then define 
\begin{equation*}
\fF(y,x,f):=\int \fO f(y)\vd x
\text{ for }(y,x,f)\in\uC_p(\uV)\times \uC_p(\uU)\times \cC^{k+\gamma}(\uV,\uW),
\end{equation*}
as the integral is well defined as a Young integral using our constraint
on $\gamma$, $\kappa$ and $p$.
The map $\fF$ is linear and continuous with respect to $(z,f)$. 
It is of class $\cC^{k+\overline{\kappa}\gamma}$ with respect to $y$.
When $1+\gamma>p$, it is well known that $y=\fF(y,z,f)$ has a unique solution
(see \textit{e.g.}, \cites{lyons02b,friz-victoir,friz14a}). Besides, 
it is evident that for any $a\in\uV$, 
\begin{gather*}
y_t=b_t+\int_{0}^t f(y_s)\vd x_s,\ t\in[0,T]\\
\text{ if and only if }
y^\star_t=b^*_t+\int_{0}^t f(a+y^\star_s)\vd x_s,\ t\in[0,T]
\end{gather*}
when $y=a+y^\star$, $b=a+b^\star$, $b^\star,y^\star\in\uCs_p(\uV)$, so that $\uCsim_p(\uV)
=\uC_p(\uV)/\uCs_p(\uV)$ is identified $\uV$ and $a=y_{0}$.

The Fréchet derivatives of $\fF$ is the direction of the variable $y$ is 
\begin{equation*}
    \dD_{y}\fF(y,z,f)\cdot h=\int (\fO\dD f(y)\cdot h)\vd z
\end{equation*}
which is also well defined as a Young integral.

From II), if $h\in\uCs_p(\uV)$, then $\dD f(y)\cdot h$
takes its values in $\uCs_p(\uV)$. 
With \eqref{eq:115}, \eqref{eq:70} and I), 
for $T$ is small enough (depending only on $\normfi{1+\gamma}{f}$ and $\normp{x}$), 
$h\mapsto \int \dD f(y)\cdot h\vd x$ is strictly contractive on $\uCs_p(\uV)$. 
This proves that $\fG(y^\star,a,x,f):=\fF(a+y^\star,x,f)$ satisfies 
the conditions of application of the Implicit Function Theorem~\ref{thm:implicit}
(See Remark~\ref{rem:implicit}). This is illustrated by Figure~\ref{fig:1}.

We then recover and extend the result in \cite{li05a}.

\begin{figure}
    \begin{center}
	\small
	\begin{tikzpicture}

	    \node[draw,rounded corners] (A) {$y\in \uC_p(\uV)$} ;

	    \node[right=1.5cm of A,draw,rounded corners] (B) {$\fO f(y)\in\uC_q(\uV)$};

	    \node[above= of B,draw,rounded corners] (C) {$x\in\uC_p(\uU)$};

	    \coordinate (BC) at ($(B)!0.5!(C)$);

	    \node[right=2.5cm of BC,draw,rounded corners] (I) {$a+\int \fO f(y)\vd x\in\uC_p(\uV)$};

	    \draw[->] ($(A.east)+(1mm,0)$) -- node[above] {$\cC^{k+\overline{\kappa}\gamma}$}  ($(B.west)-(1mm,0)$);

	    \draw[->] ($(B.east)+(1mm,0)$) -- node[below,right,xshift=2mm,pos=0.2] {$\cC^\infty$ (linear, cont.)} ($(I.west)+(-1mm,-1mm)$);

	    \draw[->] ($(C.east)+(1mm,0)$) -- node[above,right,xshift=2mm,pos=0.2] {$\cC^\infty$ (linear, cont.)} ($(I.west)+(-1mm,1mm)$);

	    \coordinate (Ifixed) at ($(I.south west)!0.75!(I.south east)$);
	    \coordinate (Ifixedsouth) at ($(Ifixed)-(0,1.5cm)$);

	    \draw[->] let \p1=(Ifixedsouth),\p2=(A.south) in ($(Ifixed)-(0,1mm)$) -- (Ifixedsouth) -- node[below,midway] {fixed point} (\x2,\y1) -- ($(A.south)+(0,-1mm)$);

	\end{tikzpicture}
	\caption{\label{fig:1} Schematic representation of the use of the Omega lemma.}
    \end{center}
\end{figure}
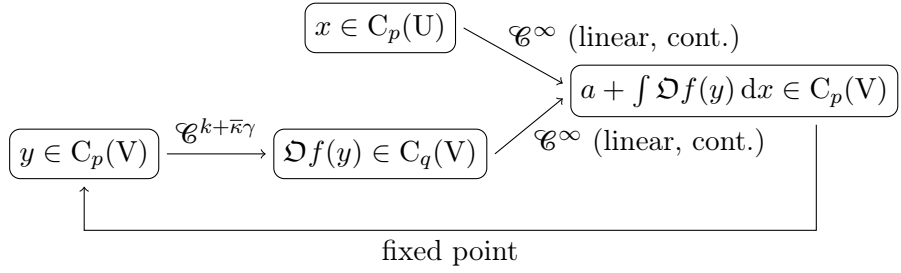

%%%%
\begin{theorem}
    \label{thm:regul:1}
    Fix $p\in[1,2)$, $x\in\uC_p(\uU)$, $a\in\uV$ 
    and $f\in\cC^{k+\gamma}(\uV,L(\uU,\uV))$, $\kappa\in(0,1)$, $\gamma\in(0,1]$, $k\geq 1$, 
   provided that $1+\kappa\gamma>p$, 
   there exists a unique solution $\fI(a,f,x,b)$ in $\uC_p(\uV)$ to 
   \begin{equation}
       \label{eq:71}
   y_{t}=a+\int_{0}^t f(y_s)\vd x_s+b_{0,t},\ t\in[0,T].
   \end{equation}
   Besides, $\fI$ is locally of class $\cC^{k+\overline{\kappa}\gamma}$
   from $\uV\times\cC^{k+\gamma}(\uV,L(\uU,\uV))\times \uC_p(\uU)\times\uC_p(\uV)$ to $\uC_p(\uV)$.
\end{theorem}

For some $t\in[0,T]$, let $\fe_t:\uC_p(\uV)\to\uV$ be the evaluation map $\fe_t(y)=y_t$.
Thanks to~\eqref{eq:70}, $\fe_t$ is continuous. We then define
$\ff_t(a):=\fe_t\circ\fI(a,f,0)$ for some vector field $f$.

\begin{corollary}
    \label{cor:flow}
    Under the conditions of Theorem~\ref{thm:regul:1}, 
    $\ff_t:\uV\to\uV$ is locally a diffeomorphism of class $\cC^{k+\overline{\kappa}\gamma}$
   for any $t\in[0,T]$. 
\end{corollary}

\begin{proof}
    Let us consider the solution to $y_{t,r}(a)=a+\int_r^t f(y_{s,r}(a))\vd x_s$
    for $0\leq r\leq t$.  Using the chain rule by combining the Omega lemma (Proposition~\ref{prop:om:fpv})
    with the bilinearity of the Young integral (Proposition~\ref{prop:young}), 
    \begin{multline}
	\label{eq:123}
	\dD_a y_{t,r}(a)=\Id+\int_r^t \dD f(y_{s,r}(a))\dD_a y_{s,r}(a)\vd x_s\\
	=\Id+\int_r^t \dD f(y_{s,r}(a))\fe_s\circ \dD_a\fI(a,f,0)\vd x_s. 
    \end{multline}
    Owing to the controls given above ---~namely \eqref{eq:115}, the boundedness of $\dD f$
    and Theorem~\ref{thm:regul:1}~--- the right-hand side of \eqref{eq:123}
    is bounded when $a$ belongs to a set $\Set{\abs{a}\leq R}$. Besides, due to \eqref{eq:70}, 
    it is easily seen that for $t-r$ small enough (in function of $f$, $\gamma$, $\kappa$, $\omega_{r,t}$, $p$
    and $R$ such that $\abs{a}\leq R$), one may choose a constant $0<\ell<1$ such that 
    \begin{equation*}
    \normop{\dD_a y_{t,r}(a)-\Id}\leq \ell,
    \end{equation*}
    where $\normop{\cdot}$ is the operator norm of $L(\uV,\uV)$.
    It follows that $\dD_a y_{t,r}$ is invertible at the point $a$.
    The Inverse Mapping Theorem (see \textit{e.g.}, \cite{abraham}*{Theorem 2.5.2, p.~116})
    and Remark~\ref{rem:implicit}
    assert that $y_{t,r}$ is then locally a $\cC^{k+\overline{\kappa}\gamma}$-diffeomorphism
    around any point $a\in\uV$.

    Using the additive property of the integral and the uniqueness of the solution to~\eqref{eq:71}, 
    $y_{t,r}(a)=y_{t,s}(y_{s,r}(a))$. Hence, the regularity of $a\mapsto y_{t,r}(a)$
    for a arbitrary times $r,t$ is treated using the stability of $\cC^{k+\overline{\kappa}\gamma}$
    under composition and the flow property.

    To conclude, it remains to remark that $\ff_t(a)= y_{t,0}(a)$.
\end{proof}

%%%%%%%%%%%%%%%%%%%%%%%%%%%%%%%%%%%%%%%%%%%%%%%%%%%%%%%%%%%%%%%%%%%%%%
\section{The Omega lemma for controlled rough paths}
%%%%%%%%%%%%%%%%%%%%%%%%%%%%%%%%%%%%%%%%%%%%%%%%%%%%%%%%%%%%%%%%%%%%%%
%%%%%%%%%%%%%%%%%%%%%%%%%%%%%%%%%%%%%%%%%%%%%%%%%%%%%%%%%%%%%%%%%%%%%%

\label{sec:crp}

Let us fix $p,q,r$ with $2\leq p<3$, $q>0$ and $0<r<p$. 
We consider a rough path~$\bx\in\uR_p(\uU)$ (see the Introduction 
for a definition).

A \emph{Controlled Rough Path} (CRP) is a path $y:[0,T]\to\uV$ which admits the
following decomposition for any $s,t$:
\begin{equation}
    \label{eq:crp}
y_{s,t}=\yd_s \bx^1_{s,t}+y^\sharp_{s,t}
\text{ with }\yd\in\uC_q(L(\uU,\uV))
\text{ and }\normr{\ys}<+\infty.
\end{equation}

\begin{notation}
A CRP $y$ is identified as the pair 
of paths $(y,\yd)$, as $\ys$ may be computed from $y$ and $\yd$. 
\end{notation}

We write
\begin{equation*}
    \normx{y}:=\normq{\yd}+\normr{\ys}
    \text{ and }
    \normfx{y}:=\abs{y_0}+\abs{\yd_0}+\normx{y}.
\end{equation*}
The space of CRP is denoted by $\uP_{p,q,r}(\bx,\uV)$. This is a Banach space
with the norm $\normfx{\cdot}$.

Useful inequalities are 
\begin{align}
    \label{eq:useful:1}
    \normsup{y^\dag}&\leq (1+\omega_{0,T}^{1/q})\normfx{y},\\
    \notag
    \normp{y}&\leq \normsup{y^\dag}\normp{\bx}+\normr{\ys}\omega_{0,T}^{1/r-1/p}\\
    \label{eq:useful:1bis}
	     &\leq \normfx{y}(1+\normp{\bx})(1+\omega_{0,T}^{1/q}+\omega_{0,T}^{1/r-1/q})
    \text{ when }r\leq p\\
    \label{eq:useful:2}
    \text{ and }
    \normsup{y}&\leq (1+\omega_{0,T}^{1/q}+\omega_{0,T}^{1/r-1/p})(1+\omega_{0,T}^{1/p})\normfx{y}(1+\normp{\bx}).
\end{align}

\begin{remark}
    \label{rem:crp}
    This definition, which involves three indices, is more general than 
    the one in \cite{friz14a}, in which $(p,q,r)=(p,p,p/2)$, and the seminal
    article \cite{gubinelli}, in which $(p,q,r)=(p,q,(p^{-1}+q^{-1})^{-1})$. 
    The reason of this flexibility will appear with the Omega lemma.
\end{remark}

For any $s\in[0,T]$, we extend $\yd_{s}$ as an operator in $L(\uU\otimes \uU,\uU\otimes \uV)$
by $\yd_{s}(a\otimes b)=a\otimes \yd_{s} b$ for all $a,b\in\uU$.

\begin{proposition} Assume that 
    \begin{equation*}
	\theta:=\min\Set*{\frac{2}{p}+\frac{1}{q},\frac{1}{p}+\frac{1}{r}}>1. 
    \end{equation*}
    There exists a continuous linear map $y\mapsto y^\flat$ on $\uP_{p,q,r}(\bx,\uV)$ 
    which transforms~$y$ to $\Set{y^\flat_{s,t}}_{[s,t]\subset[0,T]}$ with 
    values in $\uU\otimes \uV$ such that 
    \begin{equation}
	\label{eq:104}
    y^\flat_{r,s}+y^\flat_{s,t}+\bx^1_{r,s}\otimes y_{s,t}=y^\flat_{r,t},
     \text{ for all } 0\leq r\leq s\leq\ t\leq T.
    \end{equation}
    Moreover, for some universal constants $K$ and $K'$, 
    \begin{gather}
	\label{eq:103}
	\abs{y^\flat_{s,t}-\yd_s\bx^2_{s,t}}
	\leq 
	K\normx{y}(\normp{\bx}\vee\normp{\bx}^2)\omega_{s,t}^\theta\\
	\shortintertext{and}
	\label{eq:105}
	\begin{aligned}
	    \normpot{y^\flat}&\leq
	    \big(\abs{\yd_0}+(\omega_{0,T}^{1/q}+K\omega_{0,T}^{\theta-2/p}\normx{y})\big)(\normp{\bx}\vee\normp{\bx}^2)\\
			     &\leq K'\normfx{y}(\normp{\bx}\vee\normp{\bx}^2).
    \end{aligned}
    \end{gather}
\end{proposition}

\begin{proof}
    Let us introduce on the linear space $\uW:=\uU\oplus\uV\oplus(\uU\otimes\uV)$ the 
    (non-commutative) operation
    \begin{equation*}
    (a,b,c)\boxtimes (a',b',c')=(a+a',b+b',c+c'+a\otimes b')
    \end{equation*}
    and the norm $\abs{(a,b,c)}:=\max\Set{\abs{a},\abs{b},\abs{c}}$.

    For $Y,X,Z\in\uW$, it is clear that 
    $\abs{\cdot}$ is Lipschitz continuous and 
    \begin{equation*}
	\abs{Z\boxtimes Y-Z\boxtimes X}\leq (1+\abs{Z})\abs{Y-X}\text{ and } 
	\abs{Y\boxtimes Z-X\boxtimes Y}\leq (1+\abs{Z})\abs{Y-X}.
    \end{equation*}
    With $\boxtimes$, $\uW$ is a monoid for which the hypotheses
    of the Multiplicative Sewing Lemma are fulfilled \cite{feyel}.

    We then define for $y\in\uP_{p,q,r}(\bx,L(\uU,\uV))$ the family of operators
    \begin{equation*}
	\phi_{s,t}(y):=(\bx^1_{s,t},y_{s,t},\yd_s \bx^2_{s,t}),\ [s,t]\subset[0,T].
    \end{equation*}
    Thus, 
    \begin{equation}
	\label{eq:101}
	\phi_{r,s}(y)\boxtimes\phi_{s,t}(y)-\phi_{r,t}(y)
	=(0,0,
	\yd_{r,s}\bx^2_{s,t}
	+\bx^1_{r,s}\ys_{s,t}
	-\bx^1_{r,s}\otimes\yd_{r,s}\bx^1_{s,t}
	).
    \end{equation}
    The Multiplicative Sewing Lemma \cite{feyel} on $(\phi_{s,t}(y))_{[s,t]\subset[0,T]}$
    yields the existence of a family $\Set{Y_{s,t}}_{[s,t]\subset[0,T]}$ taking its values in $\uW$ 
    with $\abs{Y_{s,t}-\phi_{s,t}(y)}\leq C\omega_{s,t}^\theta$ for any $[s,t]\subset[0,T]$.
    We define $y^\flat$ as the part in $\uU\otimes\uV$ in 
    the decomposition of $Y$ as $Y_{s,t}=(\bx^1_{s,t},y_{s,t},y^\flat_{s,t})$.
    Since $Y$ satisfies $Y_{r,s}\boxtimes Y_{s,t}=Y_{r,t}$, $\yf$ satisfies~\eqref{eq:104}.
    From~\eqref{eq:101}, we easily obtain \eqref{eq:103} and \eqref{eq:105}.

    For $y,z\in\uP_{p,q,r}(\bx,\uV)$, 
    \begin{equation*}
    \phi_{r,s}(y+z)\boxtimes\phi_{s,t}(y+z)
    =\phi_{r,s}(y)\boxtimes\phi_{s,t}(y)
    +\phi_{r,s}(z)\boxtimes\phi_{s,t}(z).
    \end{equation*}
    From this additivity property, the construction of the Multiplicative
   Sewing Lemma and \eqref{eq:105}, $y\mapsto y^\flat$ is linear and continuous.  
\end{proof}

When $y$ takes its values in $L(\uV,\uW)$, $\yd$ takes its values
in $L(\uU,L(\uV,\uW))\simeq L(\uU\otimes\uV,\uW)$.

\begin{proposition} 
    \label{prop:crp}
    Fix $(p,q,r)$ and $(p,q',r')$ with $2\leq r<p<3$.
    Assume that 
    \begin{equation*}
	\widehat{\theta}:=\min\Set*{\frac{2}{p}+\frac{1}{q},\frac{1}{p}+\frac{1}{r}}>1
	\text{ and }
	\theta':=\min\Set*{\frac{2}{p}+\frac{1}{q'},\frac{1}{p}+\frac{1}{r'}}>1.
    \end{equation*}

    There exists a bilinear continuous mapping 
    \begin{equation*}
	\begin{aligned}
	    \uP_{p,q,r}(\bx,L(\uW,\uV))\times\uP_{p,q',r'}(\bx,\uW)
	    &\mapsto \uP_{p,p\vee q',p/2}(\bx,\uV)\\
	    (y,z)&\to \int y\vd z
	\end{aligned}
    \end{equation*}
    such that $\left(\int y\vd z\right)^\dag_s=y_s\zd_s$, 
    \begin{gather}
	\label{eq:106}
	\abs*{\int_s^t y_r\vd z_r-y_s z_{s,t}-\yd_s z^\flat_{s,t}}
	\leq 
	K\normx{y}\normfx{z}(1+\normp{\bx}\vee\normp{\bx}^2)\omega_{s,t}^{\widehat{\theta}}\\
	\shortintertext{and}
    \label{eq:107}
	\begin{multlined}
	\normx*{\int y\vd z}
	\leq 
	K'\normfx{y}\normfx{z}(1+\normp{\bx}\vee\normp{\bx}^2)
    \end{multlined}
    \end{gather}
    for some universal constants $K$ and $K'$.
\end{proposition}

\begin{proof} 
We set $Y_{s,t}:=y_sz_{s,t}+\yd_s z^\flat_{s,t}$. Thus, for any $r\leq s\leq t$, 
\begin{equation*}
Y_{r,s}+Y_{s,t}-Y_{r,t}
=y^\sharp_{r,s}z_{s,t}+\yd_{r,s}z^\flat_{s,t}.
\end{equation*}
The existence of the integrals follows from the Additive Sewing
Lemma \cite{feyel}. Therefore, the inequalities \eqref{eq:106} and \eqref{eq:107} are straightforward.
\end{proof}

From now, let us fix $2\leq p<3$, $q,r\geq 1$ as well as $0<\gamma\leq 1$ and  $0<\kappa<1$. 
For two Banach spaces $\uV$ and $\uW$ and a rough path
$\bx\in\uR_p(\uU)$, we set 
\begin{gather*}
    \uP\uV:=\uP_{p,q,r}(\bx,\uV),\ 
    \uPs\uV:=\Set{y\in\uP\uV\given (y_{0},\yd_{0})=(0,0)},\\
    \uQ\uW:=\uP_{p,\frac{q\vee p}{\kappa\gamma},r\vee \frac{p}{1+\kappa\gamma}}(\bx,\uW)
    \text{ and }
    \uQs\uW:=\Set{y\in\uQ\uW\given (y_{0},\yd_{0})=(0,0)}.
\end{gather*}
The spaces $\uPs\uV$ and $\uQs\uW$ are Banach sub-spaces of $\uP\uV$
and $\uQ\uW$.

\begin{proposition}[The Omega lemma for CRP]
    \label{prop:om:crp}
    Assume that $f\in\cC^{k+1+\gamma}(\uV,\uW)$.
    Then 
    $\fO f:=(f(y_t))_{t\in[0,T]}$ is locally of class $\cC^{k+\overline{\kappa}\gamma}$
    from $\uP\uV$ to $\uQ\uW$ with $\fO f(y)^\dag=\dD f(y)\yd$.
    Besides, 
    \begin{equation*}
    \dD \fO f(y)\cdot z=
    \left( \dD f(y)_t\cdot z_t
    \right)_{t\in[0,T]}\in\uQ\uW,\ \forall y,z\in\uP\uV.
    \end{equation*}
    In addition, if $z\in\uPs\uV$, then $\dD\fO f(y)z\in\uQs\uW$
    for any $y\in\uP\uV$.
\end{proposition}

\begin{proof} 
    \noindent I) 
    Let us prove first that 
    $\fO f$ maps $\uP\uV$ to $\uP_{p,\frac{p}{\gamma},r\vee \frac{p}{1+\gamma}}(\bx,\uW)$.
       As the latter space is  continuously embedded in $\uQ\uW$,
       this proves that $\fO f$ maps $\uP\uV$ to $\uQ\uW$. 

    Set for $0\leq s\leq t\leq T$, 
    \begin{gather*}
	Y_t:=f\circ y_t,\ Y^\dag_t:=\dD f(y_t)\yd_t\\
	\text{and }
	Y^\sharp_{s,t}:=\dD f(y_s)y^\sharp_{s,t}
	+\int_0^1 (\dD f(y_s+\theta y_{s,t})-\dD f(y_s))y_{s,t}\vd \theta.
    \end{gather*}
    For $Y$, the decomposition~\eqref{eq:crp} is $Y_{s,t}=Y^\dag_s x^1_{s,t}+Y^\sharp_{s,t}$.  Besides, 
    \begin{gather*}
	\abs{Y^\sharp_{s,t}}\leq \normsup{\dD f}\normx{y}\omega_{s,t}^{1/r}+\normhold{\gamma}{\dD f}\normp{y}^{1+\gamma}\omega_{s,t}^{(1+\gamma)/p},\\
      \abs{Y^\dag_{s,t}}\leq \normhold{\gamma}{\dD f}\abs{y_{s,t}}^\gamma\cdot\normsup{\yd}+\normsup{\dD f}\cdot\abs{\yd_{s,t}}.
    \end{gather*}
    We deduce that
    $Y\in\uP_{p,q\vee \frac{p}{\gamma},r\vee \frac{p}{1+\gamma}}(\bx,\uW)$ with 
    \begin{equation*}
	\normx{Y}\leq C\normfi{1+\gamma}{f}\max\Set{\normx{y},\normfx{y}^{1+\gamma}},
    \end{equation*}
    for some universal constant $C$.

    \smallskip
    \noindent II) Setting $f(y,z)=yz$ 
    (resp. $f(y,z)=y\otimes z$)
    for
    $y\in\uP L(\uW,\uV)$
    (resp.  $y\in\uP\uV$)  and  $z\in\uP\uW$
    shows that 
    $yz\in\uP\uV$ (resp. $y\otimes z\in\uP\uW\otimes\uV$).
    Moreover, $(yz)^\dag_t=\yd_t z_t+y_t \zd_t$
    (resp. $(y\otimes z)^\dag_t=\yd_t\otimes z_t+y_t\otimes \zd_t$)
    for any $t\in[0,T]$.

    In particular, $z_0=0$ and $\zd_0=0$ implies
    that $(yz)\in\uPs\uV$ (resp. $y\otimes z\in\uPs\uW\otimes\uV$)
    when $z\in\uPs\uW$.

    In addition, it is easily obtained that for the product $y\cdot z=yz$
    or $y\cdot z=y\otimes z$, 
    \begin{equation*}
	\normfx{yz}\leq C\normfx{y}\normfx{z}\normp{\bx}
    \end{equation*}
    for some universal constant $C$.

    \smallskip
    \noindent III)
    We consider that $f\in\cC^{1+\gamma}(\uV,\uW)$.
    Let $y,z\in\uP\uV$ and set $Y_t:=f(y_t)$, $Z_t:=f(z_t)$.
    According to the definition of $Z^\dag$ and $Y^\dag$, 
    \begin{multline*}
    Z_{s,t}^\dag-Y_{s,t}^\dag\\
    =(\dD f(z_t)-\dD f(z_s)-\dD f(y_t)+\dD f(y_s))\zd_t
    +(\dD f(z_s)-\dD f(y_s))\zd_{s,t}\\
    +\dD f(y_s)(\zd_{s,t}-\yd_{s,t})
    +(\dD f(y_t)-\dD f(y_s))(\zd_t-\yd_t).
    \end{multline*}
    Applying \eqref{eq-20} in Lemma~\ref{lem-1}, for $0<\kappa<1$, 
    \begin{multline*}
	\abs{Z^\dag_{s,t}-Y^\dag_{s,t}}
	\leq 
	\normhold{\gamma}{\dD f}
	\normsup{z-y}^{\overline{\kappa}\gamma}
	(\normp{z}^{\kappa\gamma}+\normp{y}^{\kappa\gamma})
	\normsup{\zd}\omega_{s,t}^{\kappa\gamma/p}
	\\
	+\normhold{\gamma}{\dD f}\normsup{z-y}^{\gamma}\normq{\zd}\omega_{s,t}^{\gamma/q}
        +\normsup{\dD f}\normq{\zd-\yd}^\gamma\omega_{s,t}^{\gamma/q}\\
	+\normhold{\gamma}{\dD f}\normp{y}^\gamma\normsup{y^\dag-z^\dag}\omega_{s,t}^{\gamma/p}.
    \end{multline*}
    From this,  and \eqref{eq:useful:1}-\eqref{eq:useful:2},
    \begin{multline}
	\label{eq:111}
	\frac{\abs{Z^\dag_{s,t}-Y^\dag_{s,t}}}{\omega_{s,t}^{\frac{\kappa\gamma}{p\vee q}}}
        \leq 
	K_1\normhold{\gamma}{\dD f}
	\normfx{z-y}^{\overline{\kappa}\gamma}
    (\normx{z}^{\kappa\gamma}+\normx{y}^{\kappa\gamma})\normfx{z}(1+\normp{\bx})^{1+\overline{\kappa}\gamma}
	\\
	+K_2\normhold{\gamma}{\dD f}\normfx{z-y}^{\gamma}\normx{z}(1+\normp{\bx})^{1+\gamma}\\
        +K_3\normsup{\dD f}\normx{z-y}^\gamma
	+K_4\normhold{\gamma}{\dD f}(1+\normp{\bx})^\gamma
	\normfx{y}^{\gamma}
	\normfx{y-z}
    \end{multline}
    for some universal constants $K_1$, $K_2$, $K_3$ and $K_4$.
    Besides,
    \begin{multline*}
    Z^\sharp_{s,t}-Y^\sharp_{s,t}=\dD f(z_s)z^\sharp_{s,t}-\dD f(y_s)y^\sharp_{s,t}
    +\int_0^1 \big(\dD f(y_s+\tau y_{s,t})-\dD f(y_s)\big)(z_{s,t}-y_{s,t})\vd \tau\\
    +\int_0^t \big(\dD f(z_s+\tau z_{s,t})-\dD f(z_s)-\dD f(y_s+\tau y_{s,t})+\dD f(y_s)\big)y_{s,t}\vd \tau.
    \end{multline*}
    With Lemma~\ref{lem-1} and \eqref{eq:useful:1},
    \begin{multline}
	\label{eq:110}
	\abs{Z^\sharp_{s,t}-Y^\sharp_{s,t}}
        \leq 
        \normx{z}\normhold{\gamma}{\dD f}\normsup{z-y}^\gamma\omega_{s,t}^{1/r}
	\\
        +\normsup{\dD f}\normx{z-y}\omega_{s,t}^{1/r}
	+\normhold{\gamma}{\dD f}\normp{y}\normp{y-z}\omega_{s,t}^{(1+\gamma)/p}\\
        +4^{\overline{\kappa}\gamma}
	\normhold{\gamma}{\dD f}
	\normp{y}
        \normsup{z-y}^{\overline{\kappa}\gamma}
        (\normp{z}^{\kappa\gamma}+\normp{y}^{\kappa\gamma})\omega_{s,t}^{(1+\kappa\gamma)/p}.
    \end{multline}
    With \eqref{eq:useful:1} and \eqref{eq:useful:2}, we deduce
that 
\begin{equation*}
    \abs{Z^\sharp_{s,t}-Y^\sharp_{s,t}}
    \leq 
    C\omega_{s,t}^{\frac{1+\kappa\gamma}{p}\wedge \frac{1}{r}},
\end{equation*}
for some constant $C$ that depends on $\omega_{0,T}$, $\normp{\bx}$, 
$\normfx{y}$, $\normfx{z}$ and the parameters $\gamma$, $\kappa$, $p$, $q$ and $r$.

With \eqref{eq:useful:2} applied to $\normsup{y-z}$, 
\eqref{eq:110} and \eqref{eq:111} could be summarized as
    \begin{equation*}
	\normx{Z-Y}
	\leq K\generalnorm{1+\gamma}{f}\normfx{z-y}^{\overline{\kappa}\gamma},
    \end{equation*}
    where $\normx{Z-Y}$ refers to  the norm in $\uQ\uW$ and 
     $K$ is a constant which depends on $\normp{x}$, $\normfx{z}$, $\normfx{y}$, $\kappa$, $\gamma$, 
    $(p,q,r)$ and $\omega_{0,T}$.

    Moreover, $\abs{f(y_0)-f(z_0)}\leq \normsup{\dD f}\abs{y_0-z_0}$
    and 
    \begin{equation*}
	\abs{f(y)^\dag_0-f(z)^\dag_0}\leq \normhold{\gamma}{\dD f}\abs{y_0-z_0}
	+\normsup{\dD f}\abs{\yd_0-\zd_0}.
    \end{equation*}
    Up to changing $K$, we get a similar inequality as 
    $\normx{Z-Y}$ is replaced by~$\normfx{Z-Y}$.

    We have then proved that $\fO f$ is locally of class
    $\cC^{\overline{\kappa}\gamma}$ from $\uP\uV$ to $\uQ\uW$. 

    \smallskip
    \noindent IV) For dealing with the general case $f\in\cC^{k+1+\gamma}(\uV,\uW)$, 
    we apply the converse of the Taylor theorem as in the proof of 
    Proposition~\ref{prop:om:fpv}, using II), the proof being in all points
    similar.
\end{proof}

When considering a fixed point for $y\mapsto \fF(y,z,f):=\int f(y)\vd z$, 
$\fF$ should map $\uP\uV$ to $\uP\uV$. Owing to Propositions~\ref{prop:crp}
and \ref{prop:om:crp}, a suitable choice is 
\begin{equation*}
q=q'=p,\ r=r'=p/2.
\end{equation*}
From now, we use these values $(q',r')=(p,p/2)$.

We now state an existence and uniqueness result for solutions of RDE.
Its proof may be found, up to a straightforward modification for dealing with $b\not=0$, 
in~\cites{gubinelli,friz14a}. 

\begin{proposition} For any $f\in\cC^{k+1+\gamma}(\uV,L(\uW,\uV))$ with $k\geq 0$, 
    any $z\in\uP\uW$ and any $a\in\uV$, $b\in L(\uW,\uV)$, then there exists
    a unique CRP $y\in\uP\uV$ which solves
    \begin{equation}
	\label{eq:124}
    y_t=a+\int_0^t f(y_s)\vd z_s+bz_{0,t}
    \text{ with }\yd_t=f(y_t)\zd_t+b\zd_t
    \end{equation}
    for any $t\in[0,T]$.
    Besides, when $\abs{a}+\abs{b}\leq R$, then 
   $y$ belongs to a closed ball of $\uP\uV$ whose radius
    depends only on $\normfi{k+1+\gamma}{f}$, $\omega_{0,T}$, $p$, $\gamma$, $\normx{z}$ 
    and $R$.
\end{proposition}

Let us set
\begin{equation*}
\uX_{k,\gamma}:=\uV\times L(\uV,\uW)\times\cC^{k+1+\gamma}(\uV,L(\uW,\uV))\times \uP\uW.
\end{equation*}
We then define the \emph{Itô map} $\fI$ as the map sending $(a,b,f,z)\in\uX_{k,\gamma}$ to $(y,\yd)$ given
by \eqref{eq:124}.
Besides, \eqref{eq:124} is equivalent in finding $\ystar\in\uPs\uV$ which solves
\begin{equation*}
    \ystar_t=\int_0^t f(\ystar_s+a+bz_{0,s})\vd z_s
\end{equation*}
and then to set $y_t:=\ystar_t+a+bz_{0,t}$.

The proof of the regularity result is now in all points similar to the one for paths of finite $p$-variation
so that we skip it.

\begin{theorem}
    \label{thm:crp} 
    Under the above conditions, the Itô map $\fI:\uX_{k,\gamma}\to\uP\uV$
    is locally of class $\cC^{k+\overline{\kappa}\gamma}$ for any $\kappa\in(0,1)$
    with $1+\kappa\gamma>p$.
\end{theorem}

For $t\in[0,T]$, let $\fe_t$ be the evaluation map $\fe_{t}(y,\yd)=y_{t}$
(this choice forces the value $\yd_t=f(y_t)$).

The proof of the next result is in all points similar to the one
of Corollary~\ref{cor:flow}.

\begin{corollary} 
    Under the hypotheses of Theorem~\ref{thm:crp}, for any $t>0$, $\ff_t(a):=\fe_t\circ \fI(a,0,f,0)$
    is locally a $\cC^{k+\overline{\kappa}\gamma}$-diffeomorphism from $\uV$ to $\uV$.
\end{corollary}

Since $\uP_{p,p,p/2}(\bx,\uV)$ is continuously embedded in $\uC_p(\uV)$,
we could consider this approach for studying the regularity 
of 
\begin{equation}
    \label{eq:120}
y_t=a+\int_0^t f(y_s)\vd z_s+\int_0^t g(y_s)\vd h_s+b_t,\ t\geq 0,
\end{equation}
where for a rough path $\bx$, $z\in\uP_{p,p,p/2}(\bx,\uW)$, 
$h\in\uC_q(\uW')$, $f$ and $g$ are maps 
respectively of class $\cC^{k+1+\gamma}$ from $\uV$ to $L(\uW,\uV)$ and 
of class $\cC^{k+\delta}$ from $\uV$ to $L(\uW',\uV)$, provided
that 
\begin{equation*}
    \frac{1}{p}+\frac{1}{q}>1,\ 1+\gamma>p\text{ and }1+\delta>q\text{ for }0<\gamma,\delta\leq 1.
\end{equation*}
The map $\fI(f,g,b,y_0,\yd_0)$ giving the solution to \eqref{eq:120}
is then of class $\cC^{k+(1-\kappa)\min\Set{\delta,\gamma}}$.

Using for $z$ the decomposition $\zd=1$ and $z^\sharp=0$,
and replacing the vector field $f$ by $\epsilon f$, it is easily 
seen that we may consider the problem 
\begin{equation*}
y^\epsilon_t=a+\int_0^t f(y_s^\epsilon)\vd \fd_\epsilon\bx_s+\int_0^t g(y_s^\epsilon)\vd h_s
\text{ for }a=(y_0,\yd_0)\text{ given}.
\end{equation*}
Asymptotic expansions in $\epsilon$ can then be performed
as in \cites{bailleul15b,inahama1,inahama07b,inahama4,inahama5}.

%%%%%%%%%%%%%%%%%%%%%%%%%%%%%%%%%%%%%%%%%%%%%%%%%%%%%%%%%%%%%%%%%%%%%%

\end{document}